\documentclass[12pt,leqno]{amsart}
\usepackage{amscd}
\usepackage{color}
\usepackage[symbol]{footmisc}
\usepackage{amssymb}
\usepackage{amsfonts}
\usepackage{latexsym}
\usepackage{verbatim}
\usepackage{multirow}
\usepackage{bbm}
\usepackage[shortlabels]{enumitem}
\usepackage[backref=page]{hyperref}
\renewcommand*{\backref}[1]{}
\renewcommand*{\backrefalt}[4]{%
	\ifcase #1 (Not cited).%
	\or        (Cited on page~#2).%
	\else      (Cited on pages~#2).%
	\fi}
\hypersetup{colorlinks=true,linkcolor=blue,citecolor=blue,urlcolor=black}

\newcommand{\R}{\mathbb{R}}
\newcommand{\C}{\mathbb{C}}

\newcommand{\rmA}{\mathrm{A}}
\newcommand{\rmB}{\mathrm{B}}
\newcommand{\rmC}{\mathrm{C}}
\newcommand{\rmD}{\mathrm{D}}
\newcommand{\rmE}{\mathrm{E}}
\newcommand{\rmF}{\mathrm{F}}
\newcommand{\rmG}{\mathrm{G}}

\newcommand{\ad}{\operatorname{ad}}

\renewcommand{\H}{\mathbb{H}}

\newcommand{\g}{\mathfrak{g}}

\newcommand{\tr}{\operatorname{tr}}

\renewcommand{\epsilon}{\varepsilon}
\renewcommand{\phi}{\varphi}
\renewcommand{\Re}{\mathrm{Re}}

\theoremstyle{plain}
\newtheorem{thm}{Theorem}[section]
\newtheorem{prop}[thm]{Proposition}
\newtheorem{lem}[thm]{Lemma}
\newtheorem{cor}[thm]{Corollary}

\theoremstyle{definition}
\newtheorem{defn}[thm]{Definition}

\newtheorem{ex}[thm]{Example}
\title{Semi-integrable almost hyperhermitian structures}

\begin{document}
 
\thanks{}
\subjclass[2020]{32Q60, 53C15, 53C26}

\address{(G. Gentili)  Université Paris-Saclay, CNRS, Laboratoire de mathématiques d’Orsay, 91405, Orsay, France.}
\email{giovanni.gentili@universite-paris-saclay.fr}

\address{(A. Moroianu) Université Paris-Saclay, CNRS, Laboratoire de mathématiques d’Orsay, 91405, Orsay, France, and Institute of Mathematics “Simion Stoilow” of the Romanian Academy, 21 Calea Grivitei, 010702 Bucharest, Romania.}
\email{andrei.moroianu@math.cnrs.fr}

\author{Giovanni Gentili and Andrei Moroianu}

\date{\today}

\begin{abstract} 
In this work, we introduce a family of almost hyperhermitian structures that we call semi-integrable. We subdivide them into four disjoint classes and show that each class is non-empty. Finally, we construct semi-integrable almost hyperhermitian structures on all reductive Lie algebras of compact type and dimension $4n$, with $n\geq 2$.
\end{abstract}

\maketitle

\section{Introduction}

An \emph{almost hypercomplex structure} on a smooth manifold of real dimension $4n$ is a reduction of the structure group of the tangent bundle to $\mathrm{GL}(n,\H)$, where $\H$ denote the quaternions. It can be equivalently interpreted as a $2$-dimensional sphere $\sf H$ of almost complex structures. The structure group reduction is often described by means of a triple $(I,J,K)$ of almost complex structures satisfying the usual identities of quaternionic units: $IJK=-\mathrm{Id}$. Such a point of view may, however, be restrictive, as it suggests that the triple $(I,J,K)$ has a preferential role, whereas this is in general not the case. 

When all the almost complex structures of the sphere are integrable one simply drops the term \emph{almost} and refers to $\mathsf{H}$ as a \emph{hypercomplex structure}. Note that this does not entail the full integrability of $\mathsf{H}$ as a $\mathrm{GL}(n,\H)$-structure, which would mean that the manifold is locally isomorphic to $\H^n$. Indeed, as Obata showed in \cite{Obata}, integrability is encoded by the vanishing of the curvature of the unique torsion-free connection that preserves $\mathsf{H}$. Integrable $\mathrm{GL}(n,\H)$-structures were studied by Sommese \cite{Sommese}.

Boyer \cite{Boyer} proved that in real dimension $4$ there are only three classes of manifolds carrying hyperhermitian structures, up to conformal equivalence: tori, K3 surfaces and certain Hopf surfaces. Hyperhermitian $4$-manifolds were also considered by Gauduchon--Tod \cite{GT} in their study of non-trivial triholomorphic Killing vector fields.

Hypercomplex geometry has been investigated at length especially in relation to compatible hyperhermitian metrics. Let $({\sf H},g)$ be an almost hyperhermitian structure. Each pair $(L,g)$, where $L\in \mathsf{H}$, lies in one of the classes in the Gray--Hervella classification of almost Hermitian structures \cite{GH}. A refinement for almost hyperhermitian structures was later provided by Martìn-Cabrera--Swann \cite{MCS}, but they focused on the choice of a generating triple $(I,J,K)$ for $\sf{H}$, showing how the Gray--Hervella classes of $(I,g),(J,g)$ condition that of $(K,g)$. On the one hand we aim to study this problem by removing the dependence on a chosen generating triple and instead investigate how the Gray--Hervella classes of two arbitrary almost complex structures in $\sf{H}$ influences those of the remaining ones. On the other hand, for the purposes of the present paper, we shall only take into account two Gray-Hervella classes: $\mathcal{W}_3\oplus \mathcal{W}_4$, which corresponds to integrability of the almost complex structure, and $\mathcal{W}_1\oplus \mathcal{W}_2 \oplus \mathcal{W}_3 $, corresponding to semi-K\"ahler structures, i.e. those whose associated fundamental form is coclosed.

For what concerns integrability, we recall the result of Obata \cite{Obata} that if two orthogonal almost complex structures $I$, $J$ in an almost hypercomplex structure $\mathsf{H}$ are integrable, then their composition is integrable as well. Yano and Ako \cite{Yano-Ako} generalised Obata's result by showing that, in fact, every element of $\mathsf{H}$ is integrable. In Theorem \ref{Prop:integrability} we drop the assumption of orthogonality of $I$ and $J$ and prove that whenever two non-proportional almost complex structures in $\mathsf{H}$ are integrable, every element of $\mathsf{H}$ is integrable as well.

We next focus on the case when $\mathsf{H}$ contains an integrable complex structure and a semi-K\"ahler one (possibly coinciding). 
The main object of investigation of our paper is thus the following class of structures:

\begin{defn}
An almost hyperhermitian structure $({\sf H},g)$ is called {\em semi-integrable} if there exist $I,L\in \mathsf{H}$, such that $I$ is integrable and $(L,g)$ is semi-K\"ahler, i.e. $\delta \omega_{L}=0$, where $\omega_L$ is the associated fundamental $2$-form.
\end{defn}

First of all, in Theorem \ref{Thm:types} we show that semi-integrable almost hyperhermitian structures can be divided into four disjoint classes:

\begin{itemize}[itemsep=2mm]
    \item Type $1$: Every almost complex structure in ${\sf H}$ is integrable and semi-K\"ahler.
    \item Type $2$: Up to a sign there is a unique integrable complex structure in ${\sf H}$ and every almost complex structure in ${\sf H}$ is semi-K\"ahler.
    \item Type $3$: Up to a sign there is a unique integrable complex structure and the semi-K\"ahler almost complex structures in ${\sf H}$ are exactly the ones orthogonal to it.
    \item Type $4$: Up to a sign there are a unique integrable complex structure and a unique semi-K\"ahler almost complex structure in ${\sf H}$ (not necessarily distinct) and they are not orthogonal to each other.
\end{itemize}

We will provide examples of semi-integrable structures for each type, showing that these classes are not empty.
%Note, for instance that hyperk\"ahler manifolds belong to the family of semi-integrable almost hyperhermitian manifolds of type 1.

A class of hyperhermitian structures $(\mathsf{H},g)$ of independent interest consists of those equipped with a connection with skew-symmetric torsion that preserves $\mathsf{H}$ and $g$, often called the Bismut connection \cite{Bismut}. We recall that the condition of having skew-symmetric torsion can be interpreted in the context of the Gray--Hervella classification as the vanishing of the component along $\mathcal{W}_2$. These structures give rise to the so-called \emph{hyperk\"ahler manifolds with torsion} (HKT), introduced by Howe--Papadopoulos \cite{HP}. Because they represent a relaxation of hyperk\"ahler structures and because of their significance within the context of supersymmetric sigma models, they have attracted considerable attention in recent years (see, e.g., \cite{AV,BDV,BF,BFG,BFGV,FG,IP,OPS,MV,W} and references therein). Whenever the torsion $3$-form of the Bismut connection is closed, these structures are called \emph{strong} HKT and were studied in several papers (see, e.g., \cite{BF,BFG,BFGV,MV,W}). In particular, with the further assumption of having parallel Bismut torsion, building on the works \cite{BPT,BFG24}, Brienza--Fino--Grantcharov--Verbitsky \cite{BFGV} proved that any simply connected compact strong HKT manifold with parallel Bismut torsion splits as the product of a compact hyperk\"ahler manifold and a compact Lie group with a bi-invariant metric. Such a result can be regarded as a particular case of the more general description provided by Moroianu--Schwahn \cite{MS} of geometries of a Riemannian connection with parallel, skew-symmetric and closed torsion.

More precisely from the analysis of Moroianu-Schwahn it follows that an almost hyperhermitian structure preserved by a metric connection with parallel, skew-symmetric and closed torsion on a complete simply connected almost hyperhermitian manifold $(M,\mathsf{H},g)$ is the product of a complete hyperk\"ahler manifold and a reductive Lie group $(G,\mathsf{H}_G,g_G)$ of compact type endowed with a bi-invariant metric. Furthermore, the torsion form is non-trivial only on the group part, where it coincides with a multiple of the canonical $3$-form. Since this splitting is preserved by each almost complex structure belonging to $\mathsf{H}$, it follows that the hyperk\"ahler factor plays no role in the semi-integrability of $(\mathsf{H},g)$, namely, $(\mathsf{H},g)$ is semi-integrable if and only if $(\mathsf{H}_G,g_G)$ is semi-integrable. In addition, $\mathsf{H}_G$ is completely determined by its restriction on the Lie algebra $\g$ of $G$, by left-translation.

Motivated by this discussion, we will carry out our investigation on reductive Lie algebras of compact type equipped with an $\ad_\g$-invariant metric. First, we show that any semi-integrable almost hyperhermitian structure on such Lie algebras is of type 3, unless the Lie algebra is abelian (Proposition \ref{Prop:type3}). Conversely, we prove in Theorem \ref{Thm:reductive} that a reductive Lie algebra $\g$ of compact type and dimension $4n$, equipped with an $\ad_\g$-invariant metric $g$ admits a semi-integrable almost hyperhermitian structure $({\sf H},g)$ if and only if $\g \neq \mathfrak{su}(2)\times \R $.

Let us also mention that left-invariant hypercomplex structures on compact Lie groups were also constructed by Spindel--Sevrin--Troost--Van Proeyen \cite{SSTvP}, and their construction was later formalised and extended to certain homogeneous spaces by Joyce \cite{Joyce}. However, these hypercomplex structures are not semi-integrable with respect to any metric, because otherwise such metric would be balanced, which is excluded by the recent work of Fino--Grantcharov--Mainenti \cite[Theorem 4.6]{FGM}.

{\bf Acknowledgments.} 
This work was partly supported by the ANR--FAPESP project ANR-21-CE40-0017 {\it Bridges}, and by the
PNRR-III-C9-2023-I8 grant CF 149/31.07.2023 {\it Conformal Aspects of Geometry and Dynamics}.

\section{Preliminaries}

\subsection{Riemannian geometry} Let $(M,g)$ be a Riemannian manifold. 
Throughout the manuscript we will tacitly raise and lower indices of tensors by means of the metric $g$. In particular we shall often identify a $2$-form $\eta \in \Lambda^{2}M$ with a skew-symmetric endomorphism, still denoted $\eta$, as follows:
\[
g(\eta(X),Y)=\eta(X,Y)\,, \qquad X,Y\in TM\,.
\]
As a special case, when $J$ is an almost Hermitian structure on $(M,g)$, we are identifying the fundamental $2$-form
\[
\omega_J(X,Y):=g(JX,Y)\,, \qquad X,Y\in TM\,,
\]
with the almost complex structure $J$ itself. In the same way, a $3$-form $\tau \in \Lambda^3M$ may be identified with a skew-symmetric tensor of type $(2,1)$ via
\[
\tau(X,Y,Z)=g(\tau_XY,Z)\,, \qquad X,Y,Z\in TM\,.
\]
We will denote with $\delta$ the codifferential, i.e. the formal adjoint of the exterior differential $d$ with respect to $g$. Given a local orthonormal frame $(e_i)$ on $M$, we recall the well-known formula
\[
\delta=-\sum_i e_i \lrcorner  \nabla_{e_i}\,,
\]
where $\nabla$ is the Levi-Civita connection of $g$.

\subsection{Almost hypercomplex structures} Let $M$ be a smooth manifold of dimension $4n$. An almost hypercomplex structure on $M$ is a reduction of the structure group of $M$ to $\mathrm{GL}(n,\mathbb{H})\subset \mathrm{GL}(4n,\mathbb{R})$. Such a structure is defined by a triple $(I,J,K)$ of almost complex structures satisfying 
\begin{equation}\label{quat}
    IJ=-JI=K\,.
\end{equation}
This triple is not unique, but the
$2$-dimensional sphere of almost complex structures on $M$,
\[
\mathsf{H}:=\{ aI+bJ+cK \mid (a,b,c)\in S^2 \}\,
\]
is uniquely determined by the almost hypercomplex structure. A set $\{I,J,K\}\subset \mathsf{H}$ satisfying \eqref{quat} will be called a {\em basis} of $\mathsf{H}$. Note that $\mathrm{SO}(3)$ acts simply transitively on the set of bases of $\mathsf{H}$.

We will henceforth use the notation $(M,\mathsf{H})$ to denote an almost hypercomplex manifold.

We say that two elements $I_1,I_2\in\sf H$ are orthogonal if $\tr(I_1I_2)=0$. For every identification of $\sf H$ with $S^2$ by means of a basis, the orthogonal of the almost complex structure $aI+bJ+cK$ corresponding to a point $(a,b,c)\in S^2$ is the great circle in the sphere $S^2$ orthogonal to the vector $(a,b,c)\in \mathbb{R}^3$.

In Section 3 we will study almost hypercomplex structures $\sf H$ containing at least two (integrable) complex structures $P,Q\in\sf H$ with $P\neq\pm Q$. We will show that in this case every element in $\sf H$ is integrable, thus extending results by Obata \cite{Obata} and by Yano and Ako \cite{Yano-Ako}.

\subsection{Almost hyperhermitian structures} Recall the following:

\begin{defn}
    An almost {\em hyperhermitian} structure on a manifold $M$ is an almost hypercomplex structure $\mathsf{H}$ together with a Riemannian metric $g$ which is almost Hermitian with respect to every almost complex structure $L\in\mathsf{H}$.
\end{defn}

%\begin{rmk}
  %  Every two elements $I_1,I_2\in \mathsf{H}$ with $I_1\neq\pm I_2$ determine $\mathsf{H}$ completely. Indeed, modulo the action of an element of $\mathrm{SO}(3)$ one can assume that $I_1=I$ and $I_2=aI+bJ$ for some constants $a,b$ with $a^2+b^2=1$ and $b>0$. Then $I_1I_2I_1+I_2=2bJ$, so one can recover $J$ by the formula $J=\lambda(I_1I_2I_1+I_2)$ with $\lambda=\sqrt{\frac1{4n}\mathrm{tr}((I_1I_2I_1+I_2)^2)}$.
    %We will therefore sometimes use the notation $(I_1,I_2)$ to denote an almost hypercomplex structure.
%\end{rmk}

If $(\mathsf{H},g)$ is an almost hyperhermitian structure, each almost complex structure $L\in \mathsf{H}$ has some Gray-Hervella type with respect to $g$. In particular, $L$ is integrable if $(L,g)$ is of type $\mathcal{W}_3\oplus\mathcal{W}_4$ (this condition actually does not depend on $g$), and $(L,g)$ is semi-Kähler if it is of type $\mathcal{W}_1\oplus\mathcal{W}_2\oplus\mathcal{W}_3$, or, equivalently, if $\delta \omega_{L}=0$.

We now make the following:

\begin{defn}
An almost hyperhermitian structure $({\sf H},g)$ is called {\em semi-integrable} if there exist $I,L\in \mathsf{H}$, such that $I$ is integrable and $(L,g)$ is semi-K\"ahler, i.e. $\delta \omega_{L}=0$.
\end{defn}

In the last sections of the paper we will focus on semi-integrable almost hyperhermitian structures. We will show that they can be divided into 4 disjoint classes, and construct examples of such structures in each class.

\section{The integrability of almost hypercomplex structures}

Let $\mathsf{H}:=\{ aI+bJ+cK \mid (a,b,c)\in S^2 \} $ be an almost hypercomplex structure on $M$. It was proved by Obata \cite{Obata} that whenever two mutually orthogonal elements, e.g. $I,J\in\sf H$ are integrable, then so is their composition $K:=IJ$. More generally, Yano and Ako \cite{Yano-Ako} showed that the integrability of $I$ and $J$ forces any almost complex structure in $\mathsf{H}$ to be integrable. 

However, from the result of Yano and Ako, it does not follow that the integrability of two (not necessarily orthogonal) almost complex structures in $\mathsf{H}$ implies the integrability of all the others. Since we are not aware of a reference proving such a fact, we present a proof below.

\begin{thm}\label{Prop:integrability}
Let $\sf H$ be an almost hypercomplex structure. If two almost complex structures $P\neq \pm Q$ in $\mathsf{H}$ are integrable, then all elements of $\sf H$ are integrable.
\end{thm}
\begin{proof} Let $g$ be a Riemannian metric compatible with $\sf H$, with Levi-Civita connection denoted $\nabla$. 
Recall (see e.g. \cite[Lemma 11.4]{M}) that an almost complex structure $L$ compatible with $g$ is integrable if and only if $E_L(X)=0$ for every $X\in TM$, where
\begin{equation}\label{el}
    E_L(X):=L\nabla_XL-\nabla_{LX}L\,.
\end{equation}

Let $(I,J,K)$ be a basis of $\sf H$ and $X\in TM$. We aim to show that $E_L(X)$ anti-commutes with $I,J,$ and $K$ for any $L\in \mathsf{H}$. It will then immediately follow that $E_L(X)=0$. Indeed, for any endomorphism $A$ that anti-commutes with $I,J,$ and $K$ we have
\[
AK=AIJ=-IAJ=IJA=KA=-AK\,,
\]
so $A=0$. 

We denote by $\mathrm{End}^-(TM)$ the set of $g$-skew-symmetric endomorphisms of $TM$. If $A,B$ are endomorphisms of $TM$, we say that $A \equiv B \mod{\mathrm{End}^-(TM)}$ if $A-B$ is skew-symmetric with respect to $g$.

Note that $I,J,K,$ and $E_L(X)$ all belong to $\mathrm{End}^-(TM)$, and two endomorphisms $A,B\in\mathrm{End}^-(TM)$ anti-commute if and only if $AB \equiv 0 \mod{\mathrm{End}^-(TM)}$.

Without loss of generality we may assume $P=I$ and $Q=aI+bJ$ for some $(a,b)\in S^1 \setminus \{(\pm 1,0)\}$. 

We shall first prove that $J$ is integrable. Let $X$ be any tangent vector. Since evidently we have $E_J(X)J+JE_J(X)=0$, we need to prove $IE_J(X)\equiv 0 \mod{\mathrm{End}^-(TM)}$ and $KE_J(X)\equiv 0 \mod{\mathrm{End}^-(TM)}$. 

Since they will be used multiple times, we prove here the following equivalences $\mod{\mathrm{End}^-(TM)}$. First, using that $E_I(X)=0$ we obtain
\[
I\nabla_{IX}J =I\nabla_{IX}(KI)=I\nabla_{IX}KI+IK\nabla_{IX}I=I\nabla_{IX}KI+ IKI\nabla_{X}I =I\nabla_{IX}KI+K\nabla_X I\,,
\]
and since $I\nabla_{IX}KI$ is skew-symmetric, we obtain
\begin{equation}\label{eqn:useful}
I\nabla_{IX}J \equiv K\nabla_X I \mod{\mathrm{End}^-(TM)}\,.
\end{equation}
If we replace $X$ with $JX$ in \eqref{eqn:useful} we also deduce
\begin{equation}\label{eqn:useful2}
I\nabla_{KX}J \equiv K\nabla_{JX}I \mod{\mathrm{End}^-(TM)}\,.
\end{equation}
From the integrability of $I$ and $Q$ we deduce the identity
\[
\begin{aligned}
0&=E_{Q}(X)\\
&=(aI+bJ)\nabla_X(aI+bJ)-\nabla_{(aI+bJ)X}(aI+bJ)\\
&=a^2E_I(X)+ab(I\nabla_XJ+J\nabla_XI-\nabla_{IX}J-\nabla_{JX}I)+b^2E_J(X)\\
&=ab(I\nabla_XJ+J\nabla_XI-\nabla_{IX}J-\nabla_{JX}I)+b^2E_J(X)\,.
\end{aligned}
\]
Thanks to \eqref{eqn:useful} we have
\[
\begin{aligned}
IE_J(X)&=-\frac{a}{b}(-\nabla_XJ+K\nabla_XI-I\nabla_{IX}J-I\nabla_{JX}I)\\
&\equiv \frac{a}{b}(I\nabla_{IX}J-K\nabla_XI) \mod{\mathrm{End}^-(TM)}\\
&\equiv 0 \mod{\mathrm{End}^-(TM)}\,,
\end{aligned}
\]
and similarly, using \eqref{eqn:useful2} we obtain
\[
\begin{aligned}
KE_J(X)&=-\frac{a}{b}(J\nabla_XJ-I\nabla_XI-K\nabla_{IX}J-K\nabla_{JX}I)\\
&\equiv \frac{a}{b}(K\nabla_{IX}J+K\nabla_{JX}I) \mod{\mathrm{End}^-(TM)}\\
&=\frac{a}{b}(K\nabla_{IX}J+I\nabla_{KX}J) \mod{\mathrm{End}^-(TM)}\\
&\equiv \frac{a}{b}IE_J(IX) \mod{\mathrm{End}^-(TM)}\\
&\equiv 0 \mod{\mathrm{End}^-(TM)}\,,
\end{aligned}
\]
where the last equality follows from the previous computation.
This shows that $E_J(X)=0$ and thus $J$ is integrable. 

With the same technique we now show that the complex structure $R=cI+dJ$ is integrable for all $(c,d)\in S^1$. Clearly
\[
E_R(X)=cd(I\nabla_XJ+J\nabla_XI-\nabla_{IX}J-\nabla_{JX}I)
\]
and thus, thanks to \eqref{eqn:useful} and \eqref{eqn:useful2}, we compute 
\[
\begin{aligned}
IE_R(X)&=cd(-\nabla_XJ+K\nabla_XI-I\nabla_{IX}J-I\nabla_{JX}I)\\
&\equiv cd(K\nabla_XI-I\nabla_{IX}J) \mod{\mathrm{End}^-(TM)}
\\
&\equiv 0 \mod{\mathrm{End}^-(TM)}\\
JE_R(X)&=cd(-K\nabla_{X}J-\nabla_{X}I-J\nabla_{IX}J-J\nabla_{JX}I)\\
&\equiv cd(-K\nabla_{X}J-J\nabla_{JX}I) \mod{\mathrm{End}^-(TM)}\\
&\equiv cd(KJ\nabla_{JX}J-\nabla_{JX}(JI)+(\nabla_{JX}J)I) \mod{\mathrm{End}^-(TM)}\\
&\equiv 0 \mod{\mathrm{End}^-(TM)}\\
KE_R(X)&=cd(J\nabla_{X}J-I\nabla_{X}I-K\nabla_{IX}J-K\nabla_{JX}I)\\
&\equiv cd(-K\nabla_{IX}J-K\nabla_{JX}I) \mod{\mathrm{End}^-(TM)}
\\
&\equiv cd(KJ\nabla_{JIX}J-I\nabla_{KX}J) \mod{\mathrm{End}^-(TM)}\\
&= 0 \mod{\mathrm{End}^-(TM)}\,.
\end{aligned}
\]
By symmetry, to conclude the proof, it is enough to show that $K$ is integrable. 
\[
\begin{aligned}
IE_K(X)&=-J\nabla_XK-I\nabla_{KX}K\\
&\equiv (\nabla_XJ)K+I\nabla_{KX}(JI) \mod{\mathrm{End}^-(TM)}
\\
&\equiv -J(\nabla_{JX}J)K+K\nabla_{KX}I   \mod{\mathrm{End}^-(TM)}
\\
&\equiv(\nabla_{JX}J)JK+J\nabla_{JX}I   \mod{\mathrm{End}^-(TM)}
\\
&\equiv (\nabla_{JX}J)I+\nabla_{JX}(JI)-(\nabla_{JX}J)I   \mod{\mathrm{End}^-(TM)}
\\
&\equiv0 \mod{\mathrm{End}^-(TM)}\\
JE_K(X)&=I\nabla_XK-J\nabla_{KX}K\\
&\equiv - (\nabla_XI)K-J\nabla_{KX}(IJ) \mod{\mathrm{End}^-(TM)}
\\
&\equiv I (\nabla_{IX}I)K+K\nabla_{KX}J \mod{\mathrm{End}^-(TM)}
\\
&=- (\nabla_{IX}I)IK-KJ\nabla_{JKX}J \mod{\mathrm{End}^-(TM)}
\\
&\equiv -I\nabla_{IX}J+I\nabla_{IX}J \mod{\mathrm{End}^-(TM)}\\
&= 0 \mod{\mathrm{End}^-(TM)}
\end{aligned}
\]
and since, obviously, $KE_K(X) \equiv 0 \mod{\mathrm{End}^-(TM)}$, we are done.
\end{proof}

\section{Semi-integrable almost hyperhermitian structures}

Recall that the Lee form of an almost Hermitian structure $(I,g)$ is defined by
$\theta_I:=I\delta I$, where as already mentioned, we identify the skew-symmetric endomorphism $I$ with the corresponding 2-form $\omega_I$.
We begin this section with the following observation regarding the Lee forms of an almost hyperhermitian structure. It can also be deduced from \cite{CS}, but we prefer to provide a direct proof here for convenience of the reader.

\begin{lem}\label{Lem:Leeforms}
Let $({\sf H},g)$ be an almost hyperhermitian structure such that $I\in {\sf H}$ is integrable. Then the Lee forms of $(L,g)$ coincide for every $L\in \mathsf{H}$ orthogonal to $I$.
\end{lem}
\begin{proof}
Let $e_i$ be a local orthonormal frame. Fix a $J\in {\sf H}$ orthogonal to $I$ and set $K:=IJ$. We have
\[
\begin{aligned}
\delta K&= -\sum_i e_i \lrcorner \nabla_{e_i} K=-\sum_i e_i \lrcorner (\nabla_{e_i} I)J-\sum_i e_i \lrcorner I(\nabla_{e_i} J)=-\sum_i e_i \lrcorner (\nabla_{e_i} I)J+I\delta J\,.
\end{aligned}
\]
We claim that the first summand in the right hand side term vanishes. Recall that $I$ is integrable if and only if the tensor $E_I$ defined in \eqref{el} vanishes. Furthermore $0=\nabla(I^2)=(\nabla I) I+ I\nabla I$ and thus
\[
\begin{aligned}
\sum_i e_i \lrcorner (\nabla_{e_i} I)J&=\sum_i Ie_i \lrcorner (\nabla_{Ie_i} I)J=-\sum_i Ie_i \lrcorner (\nabla_{e_i} I)IJ=-\sum_i e_i \lrcorner (\nabla_{e_i} I)IJI\\
&=-\sum_i e_i \lrcorner (\nabla_{e_i} I)J\,,
\end{aligned}
\]
where in the first equality we simply replaced the basis $(e_i)$ with $(Ie_i)$. Hence the claim is proved. We thus have $\delta K=I\delta J$, whence 
\begin{equation}\label{delta}
    K\delta K=KI\delta J=J\delta J\,.
\end{equation}
Now, if $L\in \mathsf{H}$ is orthogonal to $I$, it is of the form $L=aJ+bK$ with $a^2+b^2=1$. Then, using \eqref{delta} we deduce
\[
\theta_L=(aJ+bK)\delta (aJ+bK)=a^2J\delta J+ab(J\delta K+K\delta J)+b^2K \delta K=(a^2+b^2)\theta_J=\theta_J\,,
\]
concluding the proof.
\end{proof}

From Lemma \ref{Lem:Leeforms} we obtain immediately the following.

\begin{cor}\label{Cor:S1semi}
Let $({\sf H},g)$ be a semi-integrable structure and consider $I,L\in\sf H$ such that $I$ is integrable and $(L,g)$ is semi-k\"ahler. If $L$ is orthogonal to $I$, then $(L',g)$ is semi-K\"ahler for all $L'\in \mathsf{H}$ orthogonal to $I$.
\end{cor}

We will now focus on the set of semi-integrable almost hyperhermitian structures $({\sf H},g)$, and show that it can be divided into 4 distinct classes, depending on how many almost complex structures in $\sf H$ are integrable, and how many are semi-Kähler with respect to $g$.

\begin{thm}\label{Thm:types}
Every semi-integrable almost hyperhermitian structure $({\sf H},g)$ belongs to only one of the following (disjoint) classes:
\begin{itemize}[itemsep=2mm]
    \item Type $1$: Every almost complex structure in ${\sf H}$ is integrable and semi-K\"ahler.
    \item Type $2$: Up to a sign there is a unique integrable complex structure in ${\sf H}$ and every almost complex structure in ${\sf H}$ is semi-K\"ahler.
    \item Type $3$: Up to a sign there is a unique integrable complex structure and the semi-K\"ahler almost complex structures in ${\sf H}$ are exactly the ones orthogonal to it.
    \item Type $4$: Up to a sign there are a unique integrable complex structure and a unique semi-K\"ahler almost complex structure in ${\sf H}$ (not necessarily distinct) and they are not orthogonal to each other.
\end{itemize}
\end{thm}
\begin{proof}
Let $I,L\in {\sf H}$ be almost complex structures such that $I$ is integrable and $(L,g)$ is semi-K\"ahler. Let us also fix $J\in {\sf H}$ orthogonal to $I$ and denote by $K:=IJ$.

Suppose first that there is another integrable complex structure $I'\neq \pm I$ in ${\sf H}$. Then 
Theorem \ref{Prop:integrability} tells us that every complex structure in ${\sf H}$ is integrable. We need to show that in this case a semi-integrable structure is of type 1. Using the ${\rm SO}(3)$ action on ${\sf H}$, we can now assume without loss of generality that $L=J$. By Corollary \ref{Cor:S1semi}, we deduce that for every $a,b$ with $a^2+b^2=1$, the structure $(aJ+bK,g)$ is semi-K\"ahler. Another application of Corollary \ref{Cor:S1semi} with the roles of $I$ and $J$ replaced by $aJ+bK$ and $bJ-aK$ implies that for every $x,y$ with $x^2+y^2=1$, the almost Hermitian structure $(xI+y(bJ-aK),g)$ is semi-K\"ahler. Since every element of $\sf H$ can be written in this form, we obtain that $({\sf H},g)$ is of type 1.

If $({\sf H},g)$ is not of type 1, then $\pm I$ are the only integrable complex structures in ${\sf H}$. 
%Now, either the structure is of type 2, or there exists at least one $L_0\in {\sf H}$ such that $(L_0,g)$ is not semi-K\"ahler. 

Suppose first that $L$ is not orthogonal to $I$. Using the action of ${\rm SO}(3)$ on $\sf H$ we may suppose that $L=aI+bJ$ with $a^2+b^2=1$, and $a\neq 0$. Let $L'\in\sf H$ be any almost complex structure such that $(L',g)$ is semi-K\"ahler. We write $L'=a'I+b'J+c'K$ with $(a')^2+(b')^2+(c')^2=1$. Then we compute
\begin{equation}\label{eqn:dL}
0=\delta L=a \delta I+b\delta J\,,
\end{equation}
which implies $\delta I= -\frac{b}{a}\delta J$. 

If $\delta J =0$, we obtain $\delta I=0$ and $\delta K=0$, thanks to Lemma \ref{Lem:Leeforms}. By linearity we get $\delta L_0=0$ for every $L_0\in\sf H$, so $({\sf H},g)$ is of type 2. 

Assume from now on that $\delta J \neq 0$. From \eqref{delta} and \eqref{eqn:dL} we deduce
\[
0=\delta L'=a'\delta I+b'\delta J+c'\delta K=-a'\frac{b}{a} \delta J+b'\delta J-c'I\delta J\,.
\]
As $\delta J$ and $I\delta J$ are linearly independent and non-zero, we conclude that $c'=0$ and $b'=a'\frac{b}{a}$ which forces $L'$ to be proportional to $L$. Therefore in this case the semi-integrable structure is of type 4.

The only remaining case is when $L$ is orthogonal to $I$, and there exists no semi-K\"ahler structure in $\sf H$ which is not orthogonal to $I$. Then by Corollary \ref{Cor:S1semi}, for every almost complex structure $L'\in{\sf H}$ orthogonal to $I$, $(L',g)$ is semi-K\"ahler, so the structure is of type 3.
\end{proof}

We can summarize the different classes of semi-integrable almost hyperhermitian structures with the following table: 

\begin{table}[h]
\centering
\begin{tabular}{c|c|c}
 & Integrable & Semi-K\"ahler \\
\hline
Type 1 & ${\sf H} \simeq S^2$ & ${\sf H} \simeq S^2$ \\
Type 2 & $\{\pm I\}\simeq S^0$ & ${\sf H} \simeq S^2$ \\
Type 3 & $\{\pm I\} \simeq S^0 $ & $I^\perp \simeq S^1$ \\
Type 4 & $\{\pm I\}\simeq S^0$ & $\{\pm L\}\simeq S^0,\; L \not\perp I$ \\
\end{tabular}
\end{table}

In this table, each entry in the middle column represents the set of integrable elements in $\sf H$, and each entry in the right column represents the set of semi-Kähler elements in $\sf H$.

We will now construct examples of semi-integrable almost hyperhermitian structures. The first remark is that in dimension 4 there is no such structure of type 2. This fact follows from \cite[Proposition 2]{GT}. We provide a proof for the readers' convenience:

\begin{lem}
    If $({\sf H},g)$ is an almost hyperhermitian structure such that every $L\in\sf H$ is semi-Kähler and there exists $I\in\sf H$ which is integrable, then every $L\in\sf H$ is integrable.
\end{lem}

\begin{proof}
    If $M$ has dimension 4, any semi-Kähler Hermitian structure is Kähler. Thus $I$ is parallel with respect to the Levi-Civita connection $\nabla$ of $g$. Let $I,J,K$ denote a basis of $\sf H$. The corresponding 2-forms $\omega_I,\ \omega_J$ and $\omega_K$ form an orthogonal basis of the space $\Lambda^2_+(M)$ of self-dual forms, which is preserved by $\nabla$. Moreover they have constant length $\sqrt 2$ and $\nabla I=0$. This shows that there exists a 1-form $\alpha$ such that $\nabla_X\omega_J=\alpha(X)\omega_K$ for every $X\in TM$.  Since $\omega_J$ is coclosed by assumption, we get $0=\delta\omega_J=-K(\alpha)$, so $\alpha=0$, whence $J$ is $\nabla$-parallel. This shows that $({\sf H},g)$ is hyperkähler, thus proving our claim.
\end{proof}

We will now construct examples of semi-integrable almost hyperhermitian structures on $\R^4$ of the remaining types. Since every hyperkähler structure (e.g. the flat one on $\mathbb{R}^4\simeq\mathbb{H}$) is in particular semi-integrable of type 1, we just need to construct examples of type 3 and 4. 

\begin{ex}\label{ex:4.5}
    Consider the flat hyperkähler structure $(I_0,J_0,K_0,g_0)$ on $\mathbb{R}^4$, with corresponding fundamental 2-forms 
\[\omega_{I_0}=dx_1\wedge dx_2+dx_3\wedge dx_4,\qquad 
\omega_{J_0}=dx_1\wedge dx_3-dx_2\wedge dx_4,\qquad\omega_{K_0}=dx_1\wedge dx_4+dx_2\wedge dx_3\,.
\]
For every non-constant function $f\in C^\infty(\mathbb{R}^4)$, we define 
\begin{equation}\label{exf}
    I:=I_0,\qquad J:=\cos fJ_0+\sin f K_0,\qquad K:=-\sin f J_0+\cos f K_0\,.
\end{equation}
We claim that $(I,J,K,g_0)$ is a semi-integrable almost hyperhermitian structure on $\R^4$ of type 4. Since $I$ is Kähler, it suffices, by Theorem \ref{Thm:types}, to prove that $J$ is non-integrable and not semi-Kähler.
Indeed, 
\[
\delta\omega_J=\sin f J_0(\nabla f)-\cos f K_0(\nabla f)=-K(\nabla f)
\]
is not identically 0 since $f$ is non-constant. Moreover,
for every $X\in TM$ we have $\nabla_XJ=X(f)K$, so
the tensor $E_J(X)$ defined in \eqref{el} is equal to 
\[E_J(X)=J\nabla_XJ-\nabla_{JX}J=X(f)I-(JX)(f)K\,.\]
Again, since $f$ is non-constant, $E_J(X)$ cannot vanish for every $X$, so $J$ is not integrable by \cite[Lemma 11.4]{M}.
\end{ex}

\begin{ex}
Consider again the hypercomplex structure $(I,J,K)$ on $\R^4$ defined by a function $f$ as in \eqref{exf}. We claim that for $f:=x_1$ and $g=e^{x_2}g_0$ (where $g_0$ is the standard Euclidean metric), the almost hypercomplex structure $(I,J,K,g)$ is semi-integrable of type 3. Notice first that $I$ is integrable, but $(I,g)$ is not semi-Kähler since $d\omega_I=dx_2\wedge\omega_{I}\neq 0$. Moreover, for every almost Hermitian structure $(L,g_0)$, denoting by $\omega_L=g(L\cdot,\cdot)$ and $\omega_L^0=g_0(L\cdot,\cdot)$ the fundamental forms, we have that $\omega_{J_0}^0$ and $\omega_{K_0}^0$ are closed, so 
\begin{eqnarray*}
    d\omega_J^0&=&-\sin x_1 dx_1\wedge \omega_{J_0}^0+\cos x_1 dx_1\wedge \omega_{K_0}^0\\
    &=&\sin x_1dx_1\wedge dx_2\wedge dx_4+\cos x_1dx_1\wedge dx_2\wedge dx_3\,,
\end{eqnarray*}
and thus
\begin{eqnarray*}d\omega_J&=&d(e^{x_2}\omega_J^0)=e^{x_2}(d\omega_J^0+dx_2\wedge\omega_J^0)=e^{x_2}(d\omega_J^0+\cos x_1dx_2\wedge\omega_{J_0}^0+\sin x_1 dx_2\wedge\omega_{K_0}^0)\\
&=&e^{x_2}(d\omega_J^0+\cos x_1dx_2\wedge dx_1\wedge dx_3+\sin x_1 dx_2\wedge dx_1\wedge dx_4)=0\,.
\end{eqnarray*}
This shows that $(J,g)$ is non-integrable and semi-Kähler, $(I,g)$ is integrable and not semi-Kähler, whence $(I,J,K,g)$ is semi-integrable of type 3.
\end{ex}

We will now consider the higher dimensional cases $n\geq 8$. Since the product of a semi-integrable almost hyperhermitian structure with a hyperkähler 4-manifold is again semi-integrable of the same type, in order to show that the 4 classes in Theorem \ref{Thm:types} are non-empty for all $n\ge 8$, we just need to construct a semi-integrable almost hyperhermitian structure of type 2 in dimension $8$. 

\begin{ex}
    Let $(M,g)=(\R^8,g_0)$ be the flat Euclidean space. Take any function $f$ depending only on the first 4 coordinates, and define the almost hyperhermitian structure $(I,J,K)$ on $(M,g)$ whose fundamental 2-forms are 
    \begin{eqnarray*}
        \omega_I&=&dx_1\wedge dx_2+dx_3\wedge dx_4+dx_5\wedge dx_6+dx_7\wedge dx_8,\\
        \omega_J&=&dx_1\wedge dx_3-dx_2\wedge dx_4+\cos f(dx_5\wedge dx_7-dx_6\wedge dx_8)\\
        &&+\sin f(dx_5\wedge dx_8+dx_6\wedge dx_7),\\
       \omega_K&=&dx_1\wedge dx_4+dx_2\wedge dx_3-\sin f(dx_5\wedge dx_7-dx_6\wedge dx_8)\\
        &&+\cos f(dx_5\wedge dx_8+dx_6\wedge dx_7)\,.
    \end{eqnarray*}
   Clearly $I$ is integrable (since $(I,g)$ is Kähler), and by the computations in Example \ref{ex:4.5} it is easy to check that $J$ is not integrable. Moreover, $(I,g)$ is semi-Kähler (again since it is Kähler), and 
   \begin{eqnarray*}
   \delta\omega_J&=&-\sum_{i=1}^8\frac{\partial}{\partial x_i}\lrcorner \nabla_{\frac{\partial}{\partial x_i}}\omega_J\\
   &=&-\sum_{i=1}^8\frac{\partial}{\partial x_i}\lrcorner\left(\frac{\partial \cos f}{\partial x_i}(dx_5\wedge dx_7-dx_6\wedge dx_8)+\frac{\partial \sin f}{\partial x_i}(dx_5\wedge dx_8+dx_6\wedge dx_7)\right)\\
   &=&0\,,
    \end{eqnarray*}
by the choice of $f$.
\end{ex}

\section{Semi-integrable structures on reductive Lie algebras}

Throughout this section, by reductive Lie group \emph{of compact type} we mean a Lie group whose universal cover is a direct product of a compact Lie group and a Euclidean factor $\R^k$. Also, a reductive Lie algebra \emph{of compact type} is simply the Lie algebra of a reductive Lie group of compact type.

In \cite[Thm. 5.1]{MS} the following result was obtained:
\begin{thm}\label{Thm:MS}
Let $(M,g)$ be a complete simply connected Riemannian manifold of dimension $n\geq 8$, equipped with a  connection $\nabla^\tau$ with closed and parallel skew-symmetric torsion $\tau$. If $\nabla^\tau$ preserves a hypercomplex structure $\sf H$ compatible with $g$, then $(M,g)$ splits isometrically as a product of a complete hyperk\"ahler manifold $(N,g_N)$ and a reductive Lie group $G$ of compact type, equipped with a bi-invariant metric $g_G$. Every almost complex structure in $\mathsf{H}$ preserves the two factors of this decomposition. Moreover, the torsion form $\tau$ vanishes on $N$ and its restriction to $G$ is equal to $\pm\frac12\sigma\in\Omega^3G$, where $\sigma$ denotes the canonical $3$-form of $G$, defined by $\sigma(X,Y,Z)=g_G([X,Y],Z)$.
\end{thm}

Note that the connections $\nabla^\tau$ are flat for $\tau=\pm\frac12\sigma$. Their parallel vector fields are exactly the left-invariant vector fields on $G$ for $\tau=-\frac12\sigma$ and the right-invariant vector fields for $\tau=\frac12\sigma$. In addition, we may assume without loss of generality that $\tau=-\frac12\sigma$. Indeed, applying the automorphism $x\mapsto x^{-1}$ left-invariant and right-invariant vector fields are interchanged. Therefore, an almost hyperhermitian structure compatible with $g$ and preserved by a connection with parallel, skew-symmetric, and closed torsion on $(M,g)=(N,g_N)\times (G,g_G)$ is determined by the choice of three almost complex structures on the Lie algebra $\g$ of $G$, compatible with $g_G$ and satisfying the quaternionic relations \eqref{quat}, and by the choice of a hyperkähler triple on $N$. Such a structure is semi-integrable if and only if its restriction to $G$ is semi-integrable.

Summarising, in view of Theorem \ref{Thm:MS}, in order to understand semi-integrable almost hyperhermitian structures preserved by a connection with parallel, skew-symmetric, and closed torsion on simply connected complete manifolds, we only need to understand left-invariant semi-integrable almost hyperhermitian structures compatible with a bi-invariant metric on simply connected reductive Lie groups of compact type. These, in turn are determined by their restriction to the Lie algebra. We thus make the following definition:

\begin{defn}
Let $(\g,g)$ be a reductive Lie algebra of compact type endowed with an $\ad_\g$-invariant metric $g$. An almost hyperhermitian structure ${\sf H}\subset \mathrm{End}^-({\g},g)$ is called semi-integrable if there exist $I,L\in \sf H$ with $I$ integrable and $\omega_L:=g(L\cdot,\cdot)$ co-closed.
\end{defn}

By the above discussion, semi-integrable almost hyperhermitian structures on $(\g,g)$ are in one to one correspondence with left-invariant semi-integrable almost hyperhermitian structures on the simply connected Lie group $G$ with Lie algebra $\g$, endowed with the bi-invariant metric determined by $g$.

\smallskip

Let $\g$ be a reductive Lie algebra of compact type. We begin by collecting here some well known facts about roots spaces, in order to fix notations and terminology. Let $\mathfrak{t}\subseteq \mathfrak{g}$ be a Cartan subalgebra for $\g$ and $R\subset \mathfrak{t}^*$ the corresponding root system. We denote with $\g^\C=\g \otimes \C$ the complexification of $\g$ and let $\g_{\alpha}$ denote the root space of $\g^\C$ with respect to a root $\alpha \in R$. Let $R^+ \subset R$ be a system of positive roots for $\g$, namely a subset of $R$ such that $R^+ \cup (-R^+)=R$, $R^+ \cap (-R^+)=\emptyset$, $\mathrm{Span}_\C(R^+)=\mathfrak{t}^\C$ and if $\alpha,\beta \in R^+$ are such that $\alpha+\beta \in R$, then $\alpha+\beta \in R^+$. 

Recall that a \emph{Cartan-Weyl basis} for $\g^\C$ with respect to $R^+$ is a choice of generators $E_{\pm\alpha} \in \g_{\pm\alpha}$ and elements $H_\alpha\in \mathfrak{t}^\C$ for all $\alpha \in R^+$ such that
\begin{align*}
&B(E_\alpha,E_{-\alpha})=1\,, &&[E_\alpha,E_{-\alpha}]=H_\alpha\,,\\
&B(T,H_\alpha)=\alpha(T)\,, \quad \forall T\in \mathfrak{t}\,, &&[T,E_\alpha]=\alpha(T)E_{\alpha}\,, \quad \forall T\in \mathfrak{t}\,,
\end{align*}
where $B(X,Y)=\mathrm{tr}(\ad_X\ad_Y)$ is the Cartan-Killing form of $\g^\C$.

When the dimension of $\g$ is even (and thus also its rank $r$), the choice of a system of positive roots and of a complex structure on $\mathfrak{t}$ determines a left-invariant integrable complex structure $I$ on $\g$ defined by decreeing the vectors $E_\alpha$ and $E_{-\alpha}$ to be of type $(1,0)$ and $(0,1)$ respectively. This fact was first observed by Samelson \cite{Samelson} and Wang \cite{Wang} independently. Conversely, Pittie \cite{Pittie} proved that any left-invariant integrable complex structure $I$ arises in this way, namely $I$ singles out a Cartan subalgebra $\mathfrak{t}$ such that $I\vert_{\mathfrak{t}}$ is a complex structure on $\mathfrak{t}$ and it also induces an associated system of positive roots $R^+$ with respect to $\mathfrak{t}$.

\medskip
Let $G$ be the simply connected Lie group with Lie algebra $\g$. Any bi-invariant metric on $G$ is determined by an $\mathrm{Ad}_G$-invariant scalar product $g_G$ on $\g$, and a left-invariant almost hyperhermitian structure $\sf H$ on $(G,g_G)$ is determined by its restriction to $\g$, also denoted by $\sf H$. We choose a $g_G$-orthonormal basis $\{e_i\}$ of $\g$, and extend it by left-invariance to $M$. Then $\nabla_{e_i}e_j=\frac12[e_i,e_j]$, where $\nabla$ is the Levi-Civita connection of $g_G$. 
%Let $(M,{\sf H},g)$ be a hyperhermitian manifold and consider a connection
%\[
%\nabla^\tau:=\nabla+ \tau
%\]
%with parallel and closed skew-symmetric torsion $2\tau$. Here $\nabla$ is the Levi-Civita connection with respect to $g$. Assume that $\nabla^\tau$ is hyperhermitian, i.e. $\nabla^\tau L=0$ for every $L\in {\sf H}$. Fix a local orthonormal frame $e_i$ on $M$. Under these assumptions we 
For every $L\in \mathsf{H}$ we may compute the codifferential of $\omega_L$ as follows:
\begin{equation}\label{eqn:codi}
\delta \omega_L=-\sum_i e_i \lrcorner \nabla_{e_i} \omega_L=-\frac12\sum_i e_i \lrcorner (\ad_{e_i})_* \omega_L=-\frac12\sum_i [e_i,Le_i]\,,
\end{equation}
where $(\ad_{e_i})_*$ denotes the extension of $\ad_{e_i}$ to the tensor algebra of $\g$ as derivation.
Assume now that $I\in \sf H$ is integrable. Then, once we fix a Cartan-Weyl basis $(E_\alpha)_{\alpha \in R^+}$ for $\g$ and a basis $(T_k)$ for the $(1,0)$-part of the Cartan subalgebra $\mathfrak{t}$ induced by $I$ as before, the Lee form of $(L,g)$ can be written as
\begin{equation}\label{eqn:Lee}
\theta_L=L\delta \omega_L=-\Re \left( \sum_{\alpha \in R^+} L[E_\alpha, L \bar{E}_\alpha]+\sum_{k=1}^{r/2} L[T_k,L\bar{T}_k] \right)\,,
\end{equation}
where $r$ is the rank of $\g$, which is even.

\smallskip

We now observe that on (non-abelian) reductive Lie algebras of compact type, among the classes of semi-integrable almost hyperhermitian structures introduced in Theorem \ref{Thm:types} the only ones that can occur are those of type $3$.

\begin{prop}\label{Prop:type3}
A semi-integrable almost hyperhermitian structure $({\sf H},g_G)$ on a non-abelian reductive Lie algebra $\g$ of compact type is necessarily of type $3$.
\end{prop}
\begin{proof}
Let $I\in \mathsf{H}$ be integrable and $L\in \mathsf{H}$ be semi-K\"ahler. By transitivity of the $\mathrm{SO}(3)$-action on bases of {\sf H}, there exists a basis $(I,J,K)$ such that $L=aI+bJ$, for some constants $a,b$ such that $a^2+b^2=1$. As $L$ is semi-K\"ahler, we have
\begin{equation}\label{eqn:ortho}
0=\delta \omega_L=a \delta \omega_I+b\delta \omega_J\,.
\end{equation}
But now, using \eqref{eqn:Lee} we deduce
\begin{eqnarray*}
    \delta \omega_I&=&-\Re \left(\sum_{\alpha \in R^+} [E_\alpha, I \bar{E}_\alpha]+\sum_{k=1}^{r/2} [T_k,I\bar{T}_k]\right)=\Re \left(\sqrt{-1}\sum_{\alpha \in R^+} [E_\alpha, \bar{E}_\alpha]\right)\\
    &=& \Re \left(\sqrt{-1}\sum_{\alpha \in R^+} H_\alpha\right) \in \mathfrak{t} \,,
\end{eqnarray*}
because $[T_k,\bar T_k]=0$ as $\mathfrak{t}$ is abelian. On the other hand
\[
\delta \omega_J=-\Re \left( \sum_{\alpha \in R^+} [E_\alpha, J \bar{E}_\alpha]+\sum_{k=1}^{r/2} [T_k,J\bar{T}_k] \right)\in \bigoplus_{\beta \in R} \mathfrak{g}_\beta \,.
\]
Indeed, for every $\alpha \in R^+$ and $k=1,\dots, r/2$ we can write
\[
J\bar E_\alpha= F_\alpha + S_\alpha \,, \qquad J\bar T_k=F_k+S_k\,, \qquad F_\alpha, F_k \in \bigoplus_{\beta\in R^+} \mathfrak{g}_\beta\,, \,\, S_\alpha,S_k\in \mathfrak{t}\,,
\]
and thus
\[
[E_\alpha,J \bar{E}_\alpha]=[E_\alpha,F_\alpha]+[E_\alpha,S_\alpha]\in \bigoplus_{\beta \in R^+} \mathfrak{g}_\beta\,, \qquad [T_k,J \bar{T}_k]=[T_k,F_k]\in \bigoplus_{\beta \in R^+} \mathfrak{g}_\beta\,,
\]
because $[\mathfrak{t},\mathfrak{g}_\beta]\subseteq \mathfrak{g}_\beta$. 
Hence, both summands on the right-hand side of \eqref{eqn:ortho} vanish. Suppose by contradiction that $a\neq 0$. Then we must have $\delta \omega_I=0$, i.e. $\sum_{\alpha \in R^+} H_\alpha=0$, which is impossible since $G$ is not abelian. Thus $a=0$ and $L=\pm J$, concluding that the semi-integrable almost hyperhermitian structure can only be of type 3.
\end{proof}

The remaining part of the paper is devoted to show that, except for $\mathfrak{su}(2)\oplus \R$, we can always find semi-integrable almost hyperhermitian structures on any reductive Lie algebra of compact type of dimension $4n$. Before proceeding, we make the following definition:

\begin{defn}
A pair of positive roots $(\alpha,\beta)\in R^+\times R^+$ is called \emph{sum-free} if the sum $\alpha+\beta$ is not a root. In particular, if $(\alpha,\beta)$ is sum-free, then $[E_\alpha,E_\beta]=0$.
\end{defn}

As the Lie algebra of a reductive Lie group decomposes into a direct sum of its centre and its simple ideals, we first focus on proving the existence of a certain involution on compact simple Lie groups.

\begin{lem}\label{Lem:invsimple}
Let $G$ be a compact simple Lie group of rank $r$, and $R^+(G)$ a system of positive roots. Then there exists a decomposition
\[
\textstyle
R^+(G)=R_0^+(G)\coprod R_1^+(G)
\]
with
\[
|R_0^+(G)|=\begin{cases}
0  & \text{if } \dim(G)-r \equiv 0 \mod{4}\,,\\
1 & \text{if } \dim(G)-r \equiv 2 \mod{4}\,,\\
\end{cases}
\]
and an involution without fixed points $f\colon R_1^+(G) \to R_1^+(G)$ such that $(\alpha,f(\alpha))$ is sum-free for all $\alpha \in R_1^+(G)$.
\end{lem}
\begin{proof}
Note that $\dim(G)-r=2|R^+(G)|$. Therefore we will go through the list of all compact simple Lie groups discussing them according to the parity of $|R^+(G)|$. In each case, we will define the value of $f\colon R^+_1(G) \to R^+_1(G)$ using the following notation
\[
\alpha \longleftrightarrow \beta
\]
to mean $f(\alpha)=\beta$ and $f(\beta)=\alpha$.

We recall here the simple root systems following the conventions of Bourbaki \cite{Bour}:

\begin{table}[htp]
\centering
\renewcommand{\arraystretch}{1.2}
\begin{tabular}{|c|c|c|}\hline
Dynkin type&\multicolumn{2}{|c|}{Positive roots $R^+$}\\\hline\hline
$\rmA_n$&$e_i-e_j$&$1\leq i<j\leq n+1$\\\hline
\multirow{3}*{$\rmB_n$}&$e_i$&$1\leq i\leq n$\\
&$e_i-e_j$&\multirow{2}*{$1\leq i<j\leq n$}\\
&$e_i+e_j$&\\\hline
\multirow{3}*{$\rmC_n$}&$e_i-e_j$&\multirow{2}*{$1\leq i<j\leq n$}\\
&$e_i+e_j$&\\
&$2e_i$&$1\leq i\leq n$\\\hline
\multirow{2}*{$\rmD_n$}&$e_i-e_j$&\multirow{2}*{$1\leq i<j\leq n$}\\
&$e_i+e_j$&\\\hline
\multirow{3}*{$\rmE_6$}&$-e_i+e_j$&\multirow{2}*{$1\leq i<j\leq 5$}\\
&$e_i+e_j$&\\
&$\tfrac12(e_8-e_7-e_6+\sum_{i=1}^5(-1)^{\nu(i)}e_i)$&$\sum_{i=1}^5\nu(i)$ even\\\hline
\multirow{3}*{$\rmE_7$}&$-e_i+e_j$&$1\leq i<j\leq 6$ or $(i,j)=(7,8)$\\
&$e_i+e_j$&$1\leq i<j\leq 6$\\
&$\tfrac12(e_8-e_7+\sum_{i=1}^6(-1)^{\nu(i)}e_i)$&$\sum_{i=1}^6\nu(i)$ odd\\\hline
\multirow{3}*{$\rmE_8$}&$-e_i+e_j$&\multirow{2}*{$1\leq i<j\leq 8$}\\
&$e_i+e_j$&\\
&$\tfrac12(e_8+\sum_{i=1}^7(-1)^{\nu(i)}e_i)$&$\sum_{i=1}^7\nu(i)$ even\\\hline
\multirow{4}*{$\rmF_4$}&$e_i$&$1\leq i\leq 4$\\
&$e_i-e_j$&\multirow{2}*{$1\leq i<j\leq 4$}\\
&$e_i+e_j$&\\
&$\tfrac12(e_1\pm e_2\pm e_3\pm e_4)$&\\\hline
$\rmG_2$&\multicolumn{2}{|c|}{$\alpha_1,\ \alpha_2,\ \alpha_1+\alpha_2,\ 2\alpha_1+\alpha_2,\ 3\alpha_1+\alpha_2,\ 3\alpha_1+2\alpha_2$}\\\hline
\end{tabular}
\medskip
\caption{Simple root systems in the convention of \cite{Bour}. For $\rmG_2$ we express the roots in a basis of simple roots $\alpha_1,\alpha_2$.}
\label{roots}
\end{table}

{\bf Case 1.} $|R^+(G)|\equiv 0 \mod{2}$. In this case we set $R^+_0(G)=\emptyset$ hence we only need to define an involution $f$ of $R^+(G)$ such that for each each element $\alpha\in R^+(G)$, the pair $(\alpha,f(\alpha))$ is sum-free.
\begin{itemize}
\item {\bf Case $\mathrm{A}_n$ with $n\equiv 0,3 \mod{4}$.} We set
\begin{align*}
e_{2i-1}-e_j &\longleftrightarrow e_{2i}-e_j\,, && 1<2i<j\leq n+1\,,\\
e_{4i-3}-e_{4i-2} &\longleftrightarrow e_{4i-1}-e_{4i}\,, && 1\leq i\leq k \,,
\end{align*}
where $k:=[\frac{n+1}4]$.
\item {\bf Case $\mathrm{B}_{2n}$.} We define
\begin{align*}
e_i-e_{2n}& \longleftrightarrow e_i\,, && 1\leq i<2n\,,\\
e_i-e_j& \longleftrightarrow e_i+e_j\,, && 1\leq i<j<2n\,,
\end{align*}
and pair the remaining roots $e_1+e_{2n},\dots e_{2n-1}+e_{2n},e_{2n}$ arbitrarily.

\item {\bf Case $\mathrm{C}_{2n}$.} We define
\begin{align*}
e_i-e_{2n}& \longleftrightarrow 2e_i\,, && 1\leq i<2n\,,\\
e_i-e_j& \longleftrightarrow e_i+e_{j-1}\,, && 1\leq i<j<2n\,,
\end{align*}
and pair the remaining roots $e_1+e_{2},e_2+e_3\dots e_{2n-1}+e_{2n},2e_{2n}$ arbitrarily.

\item {\bf Case $\mathrm{D}_n$.} We set
\begin{align*}
e_i-e_j &\longleftrightarrow e_i+e_j\,, && 1\leq i<j \leq n\,.
\end{align*}

\item {\bf Case $\mathrm{E}_6$.} We define the pairing
\begin{align*}
e_i+e_{j} &\longleftrightarrow -e_i+e_j\,, && 1\leq i<j\leq 5\,,
\end{align*}
and pair the remaining roots $\frac{1}{2}(e_8-e_7-e_6+ \sum_{i=1}^5(-1)^{\nu(i)} e_i)$ arbitrarily.

\item {\bf Case $\mathrm{E}_8$.} We set
\[
\begin{aligned}
e_i+e_{j} &\longleftrightarrow -e_i+e_j\,, \qquad 1\leq i<j\leq 8\,,\\
\frac{1}{2}\left(e_8+ \sum_{i=1}^7(-1)^{\nu(i)} e_i\right)& \longleftrightarrow \frac{1}{2}\left(e_8-(-1)^{\nu(1)}e_1-(-1)^{\nu(2)}e_2+ \sum_{i=3}^7(-1)^{\nu(i)} e_i\right)\,.
\end{aligned}
\]
\item {\bf Case $\mathrm{F}_4$.} We define
\begin{align*}
e_i-e_{2n}& \longleftrightarrow e_i\,, && 1\leq i<2n\,,\\
e_i-e_j& \longleftrightarrow e_i+e_j\,, 
&& 1\leq i<j<2n\,,\\
\frac{1}{2}(e_1\pm e_2 \pm e_3 \pm e_4) &\longleftrightarrow \frac{1}{2}(e_1\pm e_2 \pm e_3 \pm (- e_4)) \,,
\end{align*}
and pair the remaining roots $e_1+e_{4},e_2+e_4,e_{3}+e_{4},e_{4}$ arbitrarily.

\item {\bf Case $\mathrm{G}_2$.}
There are exactly 4 possibilities. The first is
\[
\begin{aligned}
\alpha_1& \longleftrightarrow 3\alpha_1+\alpha_2\,,\\
\alpha_2& \longleftrightarrow 2\alpha_1+\alpha_2\,,\\
\alpha_1+\alpha_2&\longleftrightarrow 3\alpha_1+2\alpha_2\,. \\
\end{aligned}
\]
The second possibility is
\[
\begin{aligned}
\alpha_1& \longleftrightarrow 3\alpha_1+2\alpha_2\,,\\
\alpha_2& \longleftrightarrow 2\alpha_1+\alpha_2\,,\\
\alpha_1+\alpha_2&\longleftrightarrow 3\alpha_1+\alpha_2\,. \\
\end{aligned}
\]
The third possibility is
\[
\begin{aligned}
\alpha_1& \longleftrightarrow 3\alpha_1+\alpha_2\,,\\
\alpha_2& \longleftrightarrow \alpha_1+\alpha_2\,,\\
2\alpha_1+\alpha_2&\longleftrightarrow 3\alpha_1+2\alpha_2\,. \\
\end{aligned}
\]
The fourth possibility is
\[
\begin{aligned}
\alpha_1& \longleftrightarrow 3\alpha_1+2\alpha_2\,,\\
\alpha_2& \longleftrightarrow \alpha_1+\alpha_2\,,\\
2\alpha_1+\alpha_2&\longleftrightarrow 3\alpha_1+\alpha_2\,. 
\end{aligned}
\]
\end{itemize}

\medskip

{\bf Case 2.} $|R^+|\equiv 1 \mod{2}$. In this case we first define the involution $f$ on all but one positive root $\gamma$ and we set $R^+_0(G)=\{\gamma\}$.
\begin{itemize}
\item {\bf Case $\mathrm{A}_n$ with $n\equiv 1,2 \mod{4}$.} We define the pairing
\begin{align*}
e_{2i-1}-e_j &\longleftrightarrow e_{2i}-e_j\,, && 1<2i<j\leq n+1\,,\\
e_{4i-3}-e_{4i-2} &\longleftrightarrow e_{4i-1}-e_{4i}\,, && 1\leq i\leq k\,,
\end{align*}
where $k:=[\frac{n-1}4]$, and the remaining root is $\gamma=e_{4k+1}-e_{4k+2}$.
\item {\bf Case $\mathrm{B}_{2n+1}$.} We define the pairing
\begin{align*}
e_i-e_{2n+1}& \longleftrightarrow e_i\,, && 1\leq i<2n+1\,,\\
e_i-e_j& \longleftrightarrow e_i+e_j\,, && 1\leq i<j<2n+1\,,
\end{align*}
choose $\gamma$ to be any of the roots $e_1+e_{2n+1},\dots e_{2n-1}+e_{2n+1},e_{2n+1}$ and pair the remaining ones arbitrarily.

\item {\bf Case $\mathrm{C}_{2n+1}$.} We set
\begin{align*}
e_i-e_{2n+1}& \longleftrightarrow 2e_i\,, && 1\leq i<2n+1\,,\\
e_i-e_j& \longleftrightarrow e_i+e_{j-1}\,, &&1\leq i<j<2n+1\,, 
\end{align*}
choose $\gamma$ to be any of the roots $e_1+e_{2},e_2+e_3\dots e_{2n}+e_{2n+1},2e_{2n+1}$ and pair the remaining ones arbitrarily.

\item {\bf Case $\mathrm{E}_7$.} We define the pairing
\[
\begin{aligned}
e_i+e_{j} &\longleftrightarrow -e_i+e_j\,, \qquad 1\leq i<j\leq 6\,,\\
\frac{1}{2}\left(e_8-e_7+ \sum_{i=1}^6(-1)^{\nu(i)} e_i\right)& \longleftrightarrow \frac{1}{2}\left(e_8-e_7-(-1)^{\nu(1)}e_1-(-1)^{\nu(2)}e_2+ \sum_{i=3}^6(-1)^{\nu(i)} e_i\right)\,,
\end{aligned}
\]
and the remaining root is $\gamma=-e_7+e_8$. \qedhere
\end{itemize}
\end{proof}

We can now extend Lemma \ref{Lem:invsimple} to the general case of reductive Lie groups of compact type.

\begin{lem}\label{Lem:invcpt}
Let $G$ be a reductive Lie group of compact type, dimension $4n$, and rank $r$, and $R^+(G)$ a system of positive roots. Then there exists a decomposition
\[
\textstyle
R^+(G)=R_0^+(G)\coprod R_1^+(G)
\]
with
\[
|R_0^+(G)|=\begin{cases}
0  & \text{if } r \equiv 0 \mod{4}\,,\\
1 & \text{if } r \equiv 2 \mod{4}\,,\\
\end{cases}
\]
and an involution without fixed points $f\colon R_1^+(G) \to R_1^+(G)$ such that the pair $(\alpha,f(\alpha))$ is sum-free for all $\alpha \in R_1^+(G)$.
\end{lem}
\begin{proof}
The Lie algebra $\g$ of $G$ decomposes as
\[
\g=\mathfrak{z}(\g)\oplus \g_1 \oplus \cdots \oplus \g_s\,,
\]
where $\mathfrak{z}(\g)$ is the centre of $\g$ and $\g_i$ is the Lie algebra of a compact simple Lie group $G_i$ for all $i=1,\dots,s$. Therefore
\[
R^+(G)=\coprod_{i=1}^s R^+(G_i)\,.
\]
For all $i=1,\dots,s$ there exists a decomposition $R^+(G_i)=R_0^+(G_i)\coprod R_1^+(G_i)$ and an involution $f_i \colon R^+(G_i)\to R^+(G_i)$ with the properties stated in Lemma \ref{Lem:invsimple}. Let $\mathcal{J}$ be the set of indices in $\{1,\dots,s\}$ such that $|R^+_0(G_i)|=1$. Pick any $j_0\in \mathcal{J}$ and define
\[
R^+_0(G):=\begin{cases}
\emptyset, & \text{if } |\mathcal{J}| \equiv 0 \mod{2}\,,\\
R^+_0(G_{j_0}), & \text{if } |\mathcal{J}| \equiv 1 \mod{2}\,,
\end{cases}
\]
and of course
\[
R^+_1(G):=R^+(G)\setminus R^+_0(G)\,.
\]
We claim that $R^+_0(G)$ satisfies the requirements of the lemma. Indeed, since $\dim(G)=4n$ we have that $r$ is a multiple of $4$ if and only if the number of roots of $G$ is a multiple of $4$. This, in turn, is equivalent to $|R^+(G)|$ being even. Finally, $|R^+(G)|$ has the same parity of $|\mathcal{J}|$ concluding the proof of the claim.

Now, let
\[
\tilde{\mathcal{J}}:=\begin{cases}
\mathcal{J}, & \text{if } |\mathcal{J}| \equiv 0 \mod{2}\,,\\
\mathcal{J}\setminus \{j_0\}, & \text{if } |\mathcal{J}| \equiv 1 \mod{2}\,.
\end{cases}
\]
and choose any involution $\tilde f$ of $\tilde{\mathcal{J}}$ such that $\tilde f(l)\neq l$ for all $l\in \tilde{\mathcal{J}}$. Denote $\gamma_i $ the root in $R^+_0(G_i).$ Finally, we define the involution $f\colon R^+_1(G) \to R^+_1(G)$ by setting $f\vert_{R^+_1(G_i)}=f_i,$ for all $i=1,\dots,s,$ and $ f(\gamma_l)=\gamma_{{\tilde f}(l)},$ for any $l\in \tilde{\mathcal{J}}.$ The fact that for every $\alpha\in R^+_1(G)$ the pair $(\alpha,f(\alpha))$ is sum-free follows from Lemma \ref{Lem:invsimple} and the fact that for $i\neq j$ the pair $(\gamma_i,\gamma_j)$ is sum-free, since $\gamma_i$ and $\gamma_j$ belong to different simple summands.
\end{proof}

We are ready to prove the existence of semi-integrable almost hyperhermitian structures.

\begin{thm}\label{Thm:reductive}
Let $\g$ be a reductive Lie algebra of compact type and dimension $4n$, equipped with an $\ad_\g$-invariant metric $g_G$. Then there exists a semi-integrable almost hyperhermitian structure $({\sf H},g_G)$ on $\g$ if and only if $\g \neq \mathfrak{su}(2)\times \R $.
\end{thm}
\begin{proof}
It is known that $\g =\mathfrak{su}(2)\times \R $ cannot admit any semi-K\"ahler structure, see \cite[Theorem 4.4]{MS}. Therefore we assume $\g \neq \mathfrak{su}(2)\times \R $.

The choice of a Cartan subalgebra $\mathfrak{t}$ and a system of positive roots $R^+(G)$ on $\mathfrak{g}$ determines a (not necessarily unique) complex structure $I$ on $\g$. Let $\mathfrak{g}^\C=\g^{1,0} \oplus \g^{0,1}$ be the splitting of the complexified Lie algebra with respect to $I$. The elements in $\g^{1,0}$ are generated by $\mathfrak{t}^{1,0}$ and the positive root spaces. Any almost hyperhermitian structure ${\sf H}$ containing $I$ is completely determined by an almost Hermitian structure $J$ anti-commuting with $I$, or, equivalently, satisfying $J\mathfrak{g}^{1,0}=\mathfrak{g}^{0,1}$.

Suppose for the moment that $\mathfrak{g}\neq \mathfrak{su}(3)$; we will treat the case of $\mathfrak{su}(3)$ separately. Let $R^+(G)=R^+_0(G) \coprod R^+_1(G)$ be a decomposition of the positive roots as in Lemma \ref{Lem:invcpt} and $f\colon  R^+_1(G) \to  R^+_1(G)$ an associated involution without fixed points. We define $J$ by setting
\[
J E_\alpha:=  E_{-f(\alpha)}\,, \qquad J E_{-\alpha}:=-E_{f(\alpha)}\,, \qquad \text{for all }\alpha \in R^+_1(G)\,,
\]
and
\[
J E_\gamma:= \bar T\,, \quad JT:=-\bar E_{\gamma}\,, \qquad \gamma \in R^+_0(G)\,,
\]
where $T\in \mathfrak{t}^{1,0}$ is such that $\gamma(T)=0$ and $g_G(T,\bar T)=g_G(E_\gamma,\bar E_\gamma)$.  Note that, since we assumed $\mathfrak{g}\neq \mathfrak{su}(3)$, the Cartan subalgebra $\mathfrak{t}$ has dimension strictly bigger than $2$, and thus there always exists such a $T$. Finally, $J$ is extended to the remainder of the Cartan subalgebra arbitrarily in such a way that it is compatible with the metric $g_G$ and $J\mathfrak{t}^{1,0}=\mathfrak{t}^{0,1}$.

Evidently $J^2=-\mathrm{Id}$ and $IJ=-JI$, so $I$ and $J$ define an almost hyperhermitian structure $\mathsf{H}$ on $(\g,g)$ which extends to a left-invariant almost hyperhermitian structure on $G$. Being defined in such a way, it is straightforward to check that each term in the sum \eqref{eqn:Lee} applied to $L=J$ vanishes, thereby showing that $(\mathsf{H},g)$ forms a semi-integrable almost hyperhermitian structure.

To conclude the proof we only need to show the result for $\g = \mathfrak{su}(3)$. We consider the Cartan subalgebra $\mathfrak{t}$ of diagonal matrices. Since the bi-invariant metric on $\g$ is unique, up to rescaling, we may assume that it is $\langle \cdot,\cdot \rangle =-B$. Set $e_i\in \mathfrak{t}^*$ to be the linear form
\[
e_i\left( \begin{smallmatrix} t_1 & 0 & 0\\
0 & t_2 & 0\\
0 & 0 & t_3 \end{smallmatrix} \right):= t_i\,,
\]
for $i=1,2,3.$ A system of positive roots for $\g$ with respect to $\mathfrak{t}$ is then given by
\[
\alpha_1=e_1-e_2, \qquad \alpha_2=e_1-e_3, \qquad \alpha_3=e_2-e_3\,.
\]
Let $E_{ij}\in \g\otimes \C=\mathfrak{sl}(3,\C)$ be the matrix with $1$ in the $(i,j)$ entry and $0$ elsewhere. It is straightforward to check that $E_{ij}\in \g_{e_i-e_j}$. It is well known that the Cartan-Killing form of $\g$ and $\g\otimes \C$ is $B(X,Y)=6\mathrm{tr}(XY)$. To see this, simply observe that $\mathrm{tr}(XY)$ is bi-invariant and so it must be a constant multiple of $B$, by Schur's Lemma. The right constant can now be found by evaluation on two arbitrary elements $X,Y$. We therefore compute $B(E_{ij},E_{ji})=6$. Now, the normalized matrices
\[
E_{\alpha_1}=\frac{1}{\sqrt{6}}E_{12}, \qquad E_{\alpha_2}=\frac{1}{\sqrt{6}}E_{13}, \qquad E_{\alpha_3}=\frac{1}{\sqrt{6}}E_{23}
\]
form a Cartan-Weyl basis of $\g \otimes \C$ and satisfy the commutation relations
\[
[E_{\alpha_1},E_{\alpha_2}]=0, \qquad [E_{\alpha_1},E_{\alpha_3}]=\frac{1}{\sqrt{6}}E_{\alpha_2}, \qquad 
[E_{\alpha_2},E_{\alpha_3}]=0\,. 
\]
Let $T\in \mathfrak{t}^{1,0}$ be an element such that $\langle T,\bar T \rangle=1$, where we extended $\langle\cdot,\cdot \rangle $ to $\g \otimes \C$ bilinearly. Set
\[
JT=a\bar T+\sum_i b_{i}\bar E_{\alpha_i}, \qquad JE_{\alpha_j}=c_{j}\bar T + \sum_i d_{ji} \bar E_{\alpha_i}\,, \qquad a,b_i,c_j,d_{ji}\in \C\,,
\]
for the generic endomorphism $J$ that anticommutes with $I$. We observe that
\[
E_{\alpha_1}=\frac{1}{2\sqrt{6}}\left( E_{12}-E_{21} \right)-\frac{\sqrt{-1}}{2\sqrt{6}}\left( \sqrt{-1} E_{12}+\sqrt{-1}E_{21} \right),
\]
where each of the summands between braces on the right-hand side lies in $\mathfrak{su}(3)$. Thus
\[
\begin{split}
\langle E_{\alpha_1},\bar E_{\alpha_1}\rangle &=\frac{1}{24} \langle E_{12}-E_{21},E_{12}-E_{21} \rangle + \frac{1}{24} \langle \sqrt{-1} E_{12}+\sqrt{-1}E_{21},\sqrt{-1} E_{12}+\sqrt{-1}E_{21}\rangle\\
%& = -\frac{1}{6}B( E_{12}-E_{21},E_{12}-E_{21} \rangle - \frac{1}{6} B(\sqrt{-1} E_{12}+\sqrt{-1}E_{21},\sqrt{-1} E_{12}+\sqrt{-1}E_{21})\\
&=  -\frac{1}{4}\tr \left(( E_{12}-E_{21})^2\right)- \frac{1}{4}\tr \left((\sqrt{-1} E_{12}+\sqrt{-1}E_{21})^2\right)=\frac{1}{2}+\frac{1}{2}=1\,.
\end{split}
\]
Similarly, one sees that $\langle E_\alpha, \bar E_\alpha \rangle =1$. Therefore, in order for the metric to be compatible with $J$ we must have
\[
\begin{aligned}
b_i&=\langle E_{\alpha_i}, J T\rangle =-\langle JE_{\alpha_i},  T\rangle = c_i,\\
d_{ji}&= \langle E_{\alpha_i}, J E_{\alpha_j}\rangle=-\langle E_{\alpha_j}, J E_{\alpha_i}\rangle=-d_{ij},\\
1&=\langle T, \bar T \rangle=\langle JT, J\bar T \rangle =|a|^2+\sum_i |b_i|^2,\\
1&=\langle E_{\alpha_j}, \bar E_{\alpha_j} \rangle=\langle JE_{\alpha_j}, J\bar E_{\alpha_j} \rangle=|c_j|^2+\sum_i |d_{ij}|^2\,.
\end{aligned} 
\]
Using these identities, the condition $J^2=-{\rm Id }$ is then easily seen to be equivalent to
\[
a=0, \qquad \sum_l b_l \bar d_{kl}=0, \qquad -b_j\bar b_i+\sum_l d_{lj} \bar d_{il}=-\delta_{ij}, \qquad  i,j,k=1,2,3\,.
\]
Finally to impose that $J$ is semi-K\"ahler, by \eqref{eqn:Lee}, we must have
\[
0=\sum_i[E_{\alpha_i},J\bar E_{\alpha_i}]+[T,J\bar T]=2\sum_i \alpha_i(T)\bar b_i E_{\alpha_i}- \frac{2}{\sqrt{6}} \bar d_{13} E_{\alpha_2}\,,
\]
which holds if and only if
\[
b_1=0, \qquad b_3=0, \qquad \alpha_2(T)\bar b_2-\frac{1}{\sqrt{6}}\bar d_{13}=0\,.
\]
Summing up, the almost complex structures $I,J,IJ$ form a basis of an almost hypercomplex structure $\sf{H}$ that, together with $\langle\cdot,\cdot\rangle$, is semi-integrable if and only if
\[
JT=b_2\bar E_{\alpha_2}, \qquad JE_{\alpha_1}=-d_{13}\bar E_{\alpha_3}, \qquad JE_{\alpha_2}=-b_2 \bar T, \qquad JE_{\alpha_3}=d_{13}\bar E_{\alpha_1},
\]
\[
|b_2|^2=1, \qquad |d_{13}|^2=1, \qquad \alpha_2(T)\bar b_2-\frac{1}{\sqrt{6}}\bar d_{13}=0\,.
\]
Note that these equations can be solved if and only if
\begin{equation}\label{eqn:A2}
|\alpha_2(T)|=\frac{1}{\sqrt{6}}\,.
\end{equation}
To show that \eqref{eqn:A2} is satisfied we write
\[
T=\left( \begin{smallmatrix}  s & 0 & 0\\
0 &  t & 0\\
0 & 0 & -s-t \end{smallmatrix} \right)
\]
with $s,t \in \C$. Since $T\in \mathfrak{t}^{1,0}$, we have
\[
0=\langle T,T\rangle=6(s^2+t^2+(s+t)^2)=12(s^2+t^2+st)
\]
so $s/t$ must be a primitive third root of unity $\epsilon$. On the other hand,
\[
1=\langle T,\bar T\rangle=6(|s|^2+|t|^2+|s+t|^2)=6(|\epsilon t|^2+|t|^2+|\epsilon t+t|^2))=18 |t|^2\,,
\]
where we also used that
\[
\epsilon+ \bar \epsilon=\epsilon+ \frac{1}{\epsilon}=\frac{\epsilon^2+1}{\epsilon}=-1\,.
\]
From this, we conclude
\[
|\alpha_2(T)|^2=|2s+t|^2=|2\epsilon t+t|^2=3|t|^2=\frac{1}{6}\,,
\]
which shows that \eqref{eqn:A2} is satisfied and thus $\g$ admits a semi-integrable hyperhermitian structure, concluding the proof.
\end{proof}


\begin{thebibliography}{[99]}
\bibitem{AV}
{\sc S.~Alesker, M.~Verbitsky}: {\em Quaternionic Monge-Ampère equation and Calabi problem for HKT-manifolds
}, Israel J. Math. {\bf 176} (2010), 109--138.

\bibitem{BPT}
{\sc G.~Barbaro, F.~Pediconi, N.~Tardini}: {\em Pluriclosed manifolds with parallel Bismut
torsion} , to appear in Crelle’s Journal.

\bibitem{BDV}
{\sc M.~L.~Barberis, I.~Dotti, M.~Verbitsky}: {\em Canonical bundles of complex nilmanifolds, with applications to hypercomplex geometry }, Math. Res. Lett. {\bf 16} no. 2 (2009), 331--347.

\bibitem{BF}
{\sc M.~L.~Barberis, A.~Fino}: {\em New HKT manifolds arising from quaternionic representations}, Math. Z. {\bf 267}
 no. 3-4 (2011), 717--735.

\bibitem{Bismut}
{\sc J.~M.~Bismut}: {\em A local index theorem for non-K\"ahler manifolds}, Math. Ann. {\bf 284} (1989), 681--699.

\bibitem{Bour}
{\sc N.~Bourbaki}: \emph{Lie Groups and Lie Algebras. Chapters 4--6}, Springer 2002.

\bibitem{Boyer}
{\sc C.~P.~Boyer}: {\em A note on hyper-Hermitian four-manifolds}, Proc. Amer. Math. Soc. {\bf 102} no. 1 (1988), 157--164.

\bibitem{BFG24}
{\sc B.~Brienza, A.~Fino, G.~Grantcharov}: {\em CYT and SKT manifolds with parallel Bismut torsion}, Proc. R.
Soc. Edinb. A: Math. (2024), 1--26.

\bibitem{BFG}
{\sc B.~Brienza, A.~Fino, G.~Grantcharov}: {\em A mapping tori construction of strong HKT and generalized hyper-K\"ahler manifolds}, in \lq\lq Real and Complex Geometry - in Honour of Paul
Gauduchon\rq\rq, Springer, 2025.

\bibitem{BFGV}
{\sc B.~Brienza, A.~Fino, G.~Grantcharov, M. Verbitsky}: {\em On the structure of compact strong HKT manifolds}, Commun. Math. Phys. {\bf 407} article n. 122 (2026). 

\bibitem{CS}
{\sc M.~Cabrera, A.~Swann}: {\em Almost Hermitian structures and quaternionic geometries}, Differential Geom. Appl. {\bf 21} no. 2 (2004), 199--214.

\bibitem{FGM}
{\sc A.~Fino, G.~Grantcharov, A.~Mainenti}: {\em $p$-Kähler structures on fibrations and reductive Lie groups}, \url{https://arxiv.org/abs/2601.21849}.	

\bibitem{FG}
{\sc E.~Fusi, G.~Gentili}: {\em Special metrics in hypercomplex geometry}, Adv. Math. {\bf 496} (2026), Paper No. 111001.

\bibitem{GT}
{\sc P.~Gauduchon, K.~P.~Tod}: {\em Hyper-Hermitian metrics with symmetry}, J. Geom. Phys. {\bf 25} no. 3-4 (1998), 291--304.

\bibitem{GH} 
{\sc A.~Gray, L.~M.~Hervella}: \emph{The Sixteen Classes of Almost Hermitian Manifolds and Their Linear Invariants}, Ann. Mat. Pura Appl. {\bf 123} (1980), 35--58.

\bibitem{HP}
{\sc P.~S.~Howe, G.~Papadopoulos}: {\em Twistor spaces for HKT manifolds}, Phys. Lett. B {\bf 379} (1996), 80--86.

\bibitem{IP}
{\sc S.~Ivanov, A.~Petkov}: {\em HKT manifolds with holonomy $\mathrm{SL}(n,\H)$}, Int. Math. Res. Not. IMRN 2012 no. 16, 3779--3799.

\bibitem{Joyce}
{\sc D.~Joyce}: {\em Compact hypercomplex and quaternionic manifolds}, J. Differential Geom. \textbf{35} (1992), 743--761.

\bibitem{MCS}
{\sc F.~Martín~Cabrera, A.~Swann}: {\em Almost Hermitian structures and quaternionic geometries}, Differential Geom. Appl. {\bf 21} no. 2 (2004), 199--214.

\bibitem{MV}
{\sc R.~Moraru, M.~Verbitsky}: {\em Stable bundles on hypercomplex surfaces}, Cent. Eur. J. Math. {\bf 8} no. 2  (2010),
327--337.

\bibitem{M} 
{\sc A.~Moroianu}: \emph{Lectures on Kähler Geometry}, 
London Mathematical Society Student Texts {\bf 69}, Cambridge University Press, Cambridge, 2007. 

\bibitem{MS} 
{\sc A.~Moroianu, P.~Schwahn}: \emph{Geometries with parallel, skew-symmetric and closed torsion}, \url{http://arxiv.org/abs/2605.13227}

\bibitem{Obata}
{\sc M.~Obata}: {\em Affine connections on manifolds with almost complex, quaternion or Hermitian structure}, Japan. J. Math. {\bf 26} (1956), 43--77.

\bibitem{OPS}
{\sc L.~Ornea, Y.~S.~Poon, A.~Swann}: {\em Potential 1-forms for hyper-K\"ahler structures with torsion } Classical Quantum Gravity {\bf 20} no. 9  (2003), 1845--1856.

\bibitem{Pittie}
{\sc H.~V.~Pittie}: {\em The Dolbeault-cohomology ring of a compact, even-dimensional lie group}, Proc. Indian Acad. Sci. {\bf 98} no. 2-3 (1988), 117--152.

\bibitem{Samelson}
{\sc H.~Samelson}: {\em A class of complex analytic manifolds}, Portugaliae Math. {\bf 12} (1953), 129--132.

\bibitem{Sommese}
{\sc A.~Sommese}: {\em Quaternionic Manifolds}, Math. Ann. {\bf 212} (1975), 191--214.

\bibitem{SSTvP}
{\sc P.~Spindel, A.~Sevrin, W.~Troost and A.~Van Proeyen}: {\em Extended supersymmetric sigma models on group manifolds. I. The complex structures},
Nuclear Phys. B \textbf{308} (1988), 662--698.

\bibitem{Wang}
{\sc H.-C.~Wang}: {\em Closed manifolds with homogeneous complex structure}, Amer. J. Math. {\bf 76} (1954), 1--32.

\bibitem{W}
{\sc E.~Witten}: {\em Instantons and the large $\mathcal{N}=4$ algebra }, J. Phys. A {\bf 58} no. 3 (2025), Paper No. 035403, 68 pp.

\bibitem{Yano-Ako}
{\sc K.~Yano, M.~Ako}: {\em Integrability conditions for almost quaternion structures},  Hokkaido Math. J. {\bf 1} (1972), 63--86.

\end{thebibliography}
\end{document}